\documentclass[reqno,10pt]{amsart}

\usepackage{amsmath, amsthm, amsfonts}
\usepackage{pdfsync}
\usepackage{parskip}
\usepackage{hyperref}
\usepackage{mathrsfs}  
\usepackage{mathtools}
\usepackage{mdframed}
\usepackage{tikz}
\numberwithin{equation}{section}
\newtheorem{theorem}{Theorem}[section]
\newtheorem{prop}[theorem]{Proposition}
\newtheorem{definition}[theorem]{Definition}
\newtheorem{lem}[theorem]{Lemma}

\theoremstyle{remark}
\newtheorem{remark}[theorem]{Remark}
\newtheorem{rmk}[theorem]{Remarks}

\def\M{\mathsf{M}}
\def\sM{\mathscr{M}}
\def\p{\mathsf{p}}

\def\cH{\mathcal{H}}

\def\cE{\mathcal{E}}

\def\R{\mathbb{R}}
\def\C{\mathbb{C}}

\def\N{\mathbb{N}}
\def\Nz{\mathbb{N}_0}

\def\L{\mathcal{L}}
\def\Lis{\mathcal{L}{\rm{is}}}

\def\B{\mathbb{B}}

\def\cA{\mathcal{A}}

\def\cC{\mathcal{C}}

\def\Rp{{\rm{Re}}}

\def\RS{\mathcal{RS}}

\def\S{\mathcal{S}}

\def\sN{\mathscr{N}}
\def\Ric{{\rm Ric}}

\def\T{{\sf T}}

\def\min{{\rm min}}
\def\max{{\rm max}}
\def\dom{{\rm dom}}
\def\sN{{\sf N}}
\def\cO{\mathcal{O}}
\def\sep{s_{\varepsilon,p}}
\def\gep{\gamma_{\varepsilon,p}}

\begin{document}

\title[The harmonic map heat flow on conic manifolds]{The harmonic map heat flow on conic manifolds}

\author[Y. Shao]{Yuanzhen Shao}
\address{Department of Mathematical Sciences,
         Georgia Southern University, 
         65 Georgia Avenue, 
         P.O. Box 8104, 
         Statesboro, GA 30460, USA}
\email{yshao@georgiasouthern.edu}

\author[C. Wang]{Changyou Wang}
\address{Department of Mathematics,
         Purdue University, 
         150 N. University Street, 
         West La\-fayet\-te, IN 47907-2067, USA}
\email{wang2482@purdue.edu}

\subjclass[2010]{}

\begin{abstract}
In this article, we study the the harmonic map heat flow from a manifold with conic singularities to a closed manifold. In particular, we have proved the short time existence and uniqueness of solutions as well as the existence of global solutions into manifolds with nonpositive sectional curvature. These results are established in virtue of the maximal regularity theory on manifolds with conic singularities.
\end{abstract}
\maketitle

\section{\bf Introduction}


The construction of an $(n+1)$-dimensional conic manifold $\M$ starts from a $C^\infty$-compact manifold $(\tilde{\M}, \tilde{g})$ with boundary. Let $(\B,g_\B)=(\partial\tilde{\M}, \tilde{g}_{\partial\tilde{\M}})$.
Note that $\B$ does not need to be connected.
We equip $\M=\tilde{\M}\setminus \B$ with a smooth metric $g$ such that in a closed collar neighbourhood $(0,1]\times \B$ of the boundary
$$
g(x,y)= dx^2 + x^2 g_{\B}(y),\quad (x,y)\in (0,1]\times\B.
$$
Here $x$ is a boundary defining function of $\tilde{\M}$. Outside $(0,1]\times \B$, $g$ is equivalent to $\tilde{g}$. 
The resulting  $C^\infty$-manifold $(\M,g)$ is called a {\em conic manifold}, and 
the set $\{x=0\}$ is called the set of {\em conic singularities}.



In their seminal work \cite{EelSam64}, J. Eells and J.~H. Sampson initiated the study of the   harmonic map heat flow between closed manifolds. 
In our article, a closed manifold  always refer to one that is compact and without boundary.

The goal of this article is to study the harmonic map heat flow from an $(n+1)$-dimensional conic manifold $(\M,g)$ to an $m$-dimensional smooth closed manifold $(\sN,h)$.
We will assume that $(\sN,h)$ is isometrically embedded into a Euclidean space $\R^L$, and thus can be viewed as a Riemannian submanifold of $\R^L$. 

For each $u:\M\to\sN$, the harmonic map heat flow aims at finding $u:[0,T)\times \M\to \sN$ solving the following equation:
\begin{equation}
\label{S1: HHF}
\left\{\begin{aligned}
\partial_t u - \Delta_g u &=A_g(u)(\nabla u, \nabla u)  &&\text{on }&&\M_T;\\
u(0)&=u_0  &&\text{on}&&\M,&&
\end{aligned}\right.
\end{equation}
where $\Delta_g$ and $\nabla$ denote the Laplace-Beltrami operator on $(\M,g)$ and the gradient operator on $\M$, respectively, $A(\cdot)$ is the second fundamental form of $\sN$ in $\mathbb R^L$, $A_g(u)(\nabla u,\nabla u)=g^{\alpha\beta}A(u)(\frac{\partial u}{\partial x_\alpha}, \frac{\partial u}{\partial x_\beta})$,
and $\M_T:=(0,T)\times\M$ for some $T>0$.

The harmonic map heat flow is the negative $L_2$-gradient flow of the Dirichelet energy:
\begin{equation}
\label{S1: Energy}
E(u)=\frac{1}{2}\int_\M |\nabla u|_g^2 d v_g,
\end{equation}
where $|\nabla u|_g^2=|\nabla u|^2_{g^*\otimes h(u)} = g^{\alpha\beta}h_{ij}(u) \frac{\partial u^i}{\partial x_\alpha}\frac{\partial u^j}{\partial x_\beta}$ with $g^*$ being the induced metric on the cotangent bundle $T^*\M$ by $g$, and $(u^i)_{i=1,\cdots,m}$ and $(x_\alpha)_{\alpha=1,\cdots,n+1}$ denote the local coordinates of $\sN$ and $\M$. In addition, $v_g$ stands for the volume element of $(\M,g)$.

The critical points of \eqref{S1: Energy} are called harmonic maps. A problem of particular interest in geometric analysis is that, given any smooth $u:\M\to \sN$, can we deform $u$ into a harmonic map that is homotopic to $u$?

The harmonic map heat flow~\eqref{S1: HHF} is aiming at studying the above question by deforming  $u_0$ continuously along the flow to find the critical points of \eqref{S1: Energy}.

When both the target manifold $\sN$ and the domain manifold $\M$ are closed, J.~Eells and J.~H. Sampson proved in \cite{EelSam64} that any smooth initial date $u_0\in C^\infty(\M,\sN)$ admits a unique local in time smooth solution to \eqref{S1: HHF}. However, certain curvature condition on the target manifold $\sN$ is essential for their proof for the existence of  global smooth solutions.
In \cite{LinWang99}, F. Lin and the second author of this article have obtained an alternative condition on the existence of  global smooth solutions, which is different from the curvature condition on $\sN$.
We would like to refer the reader to \cite{LinWang08} for a thorough survey of the problem between compact manifolds.

R. Hamilton studied the harmonic map heat flow on compact manifolds with smooth boundary in his book \cite{Ham84}. The study of \eqref{S1: HHF} on complete and noncompact domain manifolds was initiated by R. Schoen and S.-T Yau \cite{SchYau76}. In \cite{SchYau76}, the authors showed that any $C^1$-map with finite total energy from a complete noncompact $\M$ into a complete manifold $\sN$ with nonpositive sectional curvature can be deformed into a harmonic map along the flow \eqref{S1: HHF}. 
Later, many authors considered this problem under various conditions on the complete noncompact domain manifold $\M$ and the initial datum $u_0$.
Interested reader may refer to \cite{FarReg00, Li93, LiTam91,Wang94, MWang08} for more details of the harmonic map heat flow on complete noncompact domain manifolds.

As far as we know, our paper is the first one on the study of the harmonic map heat flow on domain manifolds with singularities. The main tool in our paper is the maximal $L_p$-regularity theory for nonlinear parabolic equations. 
Our main results are the following two theorems.
 

\begin{theorem}
\label{Main theorem}
Suppose that $(\M,g)$ is an $(n+1)$-dimensional conic manifold  and $(\sN,h)$ is an $m$-dimensional smooth closed manifold.
Assume that 
\begin{itemize}
\item $\gamma$ and $p$ satisfy (B0) and (B3) in Section~\ref{Section 4} when $n\geq 3$; or
\item when $n=1,2$,  $\gamma$ and $p$ satisfy (A) in Section~\ref{Section 4} or satisfy (B0)-(B1) in Section~\ref{Section 4} for $n=1$ or  satisfy (B0) and (B2) in Section~\ref{Section 4} for $n=2$. 
\end{itemize}
Let 
$$
u_0=w_0+\p\quad \text{with}\quad \p\in\sN \quad \text{and} \quad w_0\in \cH_p^{2-\frac{2}{p}+\delta, \gamma+2-\frac{2}{p}+\delta}(\M,\sN-\p)
$$ 
for an arbitrary positive constant $\delta$. Then \eqref{S1: HHF} with initial condition $u_0$ has a unique solution
$$
u\in L_p(J_T, \cH^{2,\gamma+2}_p(\M,\R^L)\oplus \R^L)\cap H^1_p(J_T,\cH^{0,\gamma}_p(\M,\sN))
$$
on $J_T=[0,T)$ for some $T>0$. Moreover, 
$
u\in  C^\infty((0,T)\times \M, \sN).
$
\end{theorem}
The space $\cH^{s,\gamma}_p(\M,\R^L)$ can be viewed as a Bessel potential space with weight, called Mellin Sobolev space, which will be introduced in Section~\ref{Subsection 2.1}.

Let $\omega$ be a cut-off function on $[0,1)$, more precisely, $\omega\in C^\infty([0,1),[0,1])$ with $\omega\equiv 1$ near $0$ and $\omega\equiv 0$ close to $1$.

For the global existence, we have the following result.
\begin{theorem}
\label{Main theorem-2}
Suppose that the sectional curvature  of $\sN$ is nonpositive.
Assume that 
\begin{itemize}
\item $\gamma$ and $p$ satisfy (B0) and (B3) in Section~\ref{Section 4} when $n\geq 3$; or
\item when $n=1,2$,  $\gamma$ and $p$ satisfy (A') in Section~\ref{Section 5} or satisfy (B0)-(B1) in Section~\ref{Section 4} for $n=1$ or  satisfy (B0) and (B2) in Section~\ref{Section 4} for $n=2$. 
\end{itemize}
Let  
$$
u_0=w_0+\p\quad \text{with}\quad \p\in\sN \quad \text{and} \quad w_0\in \cH^{2-\frac{2}{p}+\delta, \gamma+2-\frac{2}{p}+\delta}_p(\M,\sN-\p)
$$ 
for an arbitrary positive constant $\delta$; in addition $\|\nabla w_0\|_\infty<\infty$. Then \eqref{S1: HHF} with initial condition $u_0$ has a unique global solution
$$
u\in L_p([0,T), \cH^{2,\gamma+2}_p(\M,\R^L)\oplus \R^L)\cap H^1_p([0,T),\cH^{0,\gamma}_p(\M,\sN) )
$$
for any $T>0$.
Moreover, $
u \in C^\infty((0,\infty)\times \M, \sN).
$
\end{theorem}
We would like to point out that in Theorem \ref{Main theorem-2}, the necessity of replacing Condition (A) in Theorem~\ref{Main theorem} by the stronger Condition (A') stems from the simple fact
that any nonzero constant function belongs to $\mathcal{H}^{0,\gamma}_p(\M)$ for $p> 1$ iff $\gamma<\frac{n+1}2$.


One of the essential difficulties in the analysis of the Laplace-Beltrami operator on conic manifolds is that there is no canonical choice for a closed extension  of this operator, as noticed first by J. Br\"uning and R. Seeley \cite{BruSee88}.
Indeed,  the domains of the minimal and the maximal extensions of $\Delta_g$ differ by a non-trivial  finite dimensional space. 
Functions in this finite dimensional space admit certain asymptotic behaviors as $x\to 0^+$, as we will see in Section~\ref{Subsection 3.2}.
We would like to refer the readers to \cite{CorSchSei03, Les97, RoiSch14, RoiSch15,SchSei05} for more details. 

This article is organized as follows. In Sections~\ref{Subsection 2.1}  and \ref{Subsection 2.2}, we will introduce the precise definitions and fundamental properties of the function spaces on conic manifolds. In Section~\ref{Subsection 2.3}, we will establish the point-wise multiplication and Nemyskii operator theory for the function spaces defined in Sections~\ref{Subsection 2.1}  and \ref{Subsection 2.2}.
Section~\ref{Section 3} provides a review of some necessary results on the maximal $L_p$-regularity theory and closed extensions of the Laplace-Beltrami operator on conic manifolds.
In Section~\ref{Section 4}, we state our choice for the close extensions of the Laplace-Beltrami operator and prove the local existence and uniqueness of solutions to the harmonic map heat flow by means of the results  presented in Sections~\ref{Subsection 2.3} and \ref{Section 3}. 
In Section~\ref{Section 5}, we derive a gradient estimate for the solution to the harmonic map heat flow; and  then we utilize the maximal regularity theory to obtain some higher order estimates, which gives us the benefit of establishing a global solution.

\textbf{Notations:} 

For any two Banach spaces $X,Y$, $X\doteq Y$ means that they are equal in the sense of equivalent norms. The notation $\L^j(X,Y)$ means the set of all bounded $j$-linear maps from $X$ to $Y$. In particular, we use the notation $\L(X,Y)=\L^1(X,Y)$. Moreover, $\Lis(X,Y)$ stands for the subset of $\L(X,Y)$ consisting of all bounded linear isomorphisms from $X$ to $Y$. 

Given any Banach space $X$ and a manifold $\mathscr{M}$,
let $\| \cdot \|_{k,\infty;c}$ and $\|\cdot \|_{s,\gamma;p}$ denote the norms of the $X$-valued Banach spaces $\cC_g^k(\mathscr{M},X)$ and $\cH^{s,\gamma}_p(\mathscr{M},X)$, respectively. We will introduce these spaces in Section 2. 
Meanwhile, $\| \cdot \|_{k,\infty}$ and $\|\cdot\|_{s,p}$ stand for the usual norms of the spaces $BC^k(\mathscr{M},X)$ and $H^s_p(\mathscr{M},X)$.

If the letter $X$ is omitted in the definitions of these spaces, e.g., $\cC_g(\sM)$, it means that the corresponding space is $\C$-valued.


$\R_+$ denotes $(0,\infty)$ and $I=(0,1]$. In addition, $\Nz=\N\cup \{0\}$.

Throughout the rest of this paper, unless stated otherwise, we always assume that 
\smallskip
$$\fbox{$s\geq 0$, $k\in\Nz$, $1<p<\infty$ and $\gamma\in \R$.}$$


\section{\bf Function spaces on conic manifolds}
\label{Section 2}


\subsection{\bf Mellin Sobolev spaces}\label{Subsection 2.1}
In this subsection, we describe a scale of weighted Sobolev spaces $\cH^{s,\gamma}_p (\M)$. These weighted Sobolev spaces are widely used in the analysis on conic manifolds. See \cite{Les97, RoiSch14, RoiSch15}.

We pick a cut-off function $\omega$ on $[0,1)$, which means $\omega\in C^\infty([0,1),[0,1])$ with $\omega\equiv 1$ near $0$ and $\omega\equiv 0$ near $1$.

For $k\in\Nz$,  $\cH^{k,\gamma}_p(\M)$ is the space of all functions $u\in H^k_{p,loc}(\M)$ such that near the conic singularities, or more precisely, in $I\times \B$
$$
x^{\frac{n+1}{2}-\gamma}(x \partial_x)^j \partial^\alpha_y (\omega u)  \in L_p((0,1]\times \B, \frac{dx}{x}dy),\quad j+|\alpha|\leq k,
$$
where $(x,y)\in I\times \B$ and $\alpha\in \N^n_0$.
Here $\partial_y^\alpha$ can be considered as the derivatives in local coordinates of $\B$, and we will use this slight abuse of notation in the sequel.

We also put
$$
\cH^{\infty,\gamma}_p(\M):=\bigcap _{k\in\Nz} \cH^{k,\gamma}_p(\M).
$$

To understand the motivation of this somewhat unusual definition, let us consider the flat cone $\M= I\times S^n$ in $\R^{n+1}$, where $S^n$ is the $n$-sphere. Taking polar coordinates in $\M$, then $\cH^{0,0}_2(\M)$ coincides with the usual $L_2(B_1)$ space ($B_1$ is the closed unit ball in $\R^{n+1}$).

For arbitrary $\gamma\in\R$, define the map 
$$
S_\gamma: C_c^\infty(I\times \B) \to C_c^\infty(\overline{\R}_+\times \B): \, u(x,y)\mapsto e^{(\gamma-\frac{n+1}{2})t} u(e^{-t},y).
$$
Then for any $s\geq 0$ and $\gamma\in\R$,  $\cH^{s,\gamma}_p(\M)$ is the space of all distributions on $\M$ such that
$$
\|u\|_{\cH^{s,\gamma}_p(\M)}=\|S_\gamma ( \omega u )\|_{H^s_p(\overline{\R}_+\times\B)} +  \| (1- \omega) u \|_{H^s_p(\M)} <\infty.
$$
It is understood that $\omega$ is extended to be zero outside $(0,1)\times \B$.

In the sequel, we always assume that $\overline{\R}_+ \times \B$ is equipped with the product metric $\bar{g}=dt^2+g_\B$, and $I \times \B$ is equipped with $g=dx^2 + x^2 g_\B$.

\begin{lem}
\label{S2.1: Sobolev-iso}
\begin{itemize}
\item[ ]
\item[(i) ]For all $s\geq 0$, $
S_\gamma \in \Lis(\cH^{s,\gamma}_p(I\times \B), H^s_p(\overline{\R}_+ \times \B)).
$
\item[(ii) ]Let $1<q<\infty$. Assume that $s_1 -\frac{n+1}{p}> s_0- \frac{n+1}{q}$ and $\gamma_1\geq \gamma_0$. Then the embedding 
$$
\cH^{s_1,\gamma_1}_p(\M)\hookrightarrow \cH^{s_0,\gamma_0}_q(\M).
$$
is compact.
\end{itemize}
\end{lem}
\begin{proof}
(i) easily follows from the definition of Mellin Sobolev spaces.

(ii) First it is easy to see that $
\cH^{s_1,\gamma_1}_p(\M)\hookrightarrow \cH^{s_1,\gamma_0}_p(\M).
$
Then (ii) is a direct consequence of (i) and the classical embedding results for Bessel potential spaces.
\end{proof}

For $0<\theta <1$, we denote by  $[\cdot,\cdot]_\theta$ the complex interpolation method, cf. \cite[Example I.2.4.2]{Ama95}, and
by  $(\cdot,\cdot)_{\theta,q}$ with $1\leq q \leq \infty$ the real interpolation method, cf. \cite[Example I.2.4.1]{Ama95}.

It was shown in \cite[Lemma~3.7]{RoiSch15} that, for Mellin-Sobolev spaces of the same weight, any complex interpolation will lead to another Mellin-Sobolev space. 
\begin{lem}
\label{S2.1: Sobolev-interpolation-complex}
Suppose that $0\leq s_0<s_1$ and $\gamma\in\R$.
For any $0<\theta<1$,
$$
[\cH^{s_0,\gamma}_p(\M), \cH^{s_1,\gamma}_p(\M)]_\theta  \doteq \cH^{s,\gamma}_p(\M)
$$
where   $s=(1-\theta)s_0+\theta s_1$.
\end{lem}

Whenever it causes no confusion, in the sequel, we will denote by $\C$ the space of all constant functions on $\M$.
\begin{lem}
\label{S2.1: Sobolev-interpolation}
Suppose that $0\leq s_0<s_1$ and $\gamma_0,\gamma_1\in\R$.
For any $ \varepsilon>0$ and $0<\theta<1$,
$$
\cH^{s+\varepsilon,\gamma+\varepsilon}_p(\M) \hookrightarrow  (\cH^{s_0,\gamma_0}_p(\M), \cH^{s_1,\gamma_1}_p(\M) )_{\theta,p}   \hookrightarrow \cH^{s-\varepsilon,\gamma-\varepsilon}_p(\M) ,
$$
where  $s=(1-\theta)s_0+\theta s_1$ and $\gamma=(1-\theta)\gamma_0 +\theta \gamma_1$, and for any $s\geq 0$
\begin{align*}
\cH^{s+2\theta+\varepsilon,\gamma+2\theta+\varepsilon}_p(\M)\oplus \C &\hookrightarrow  (\cH^{s,\gamma}_p(\M), \cH^{s+2,\gamma+2}_p(\M)\oplus\C)_{\theta,p}\\
   &\hookrightarrow \cH^{s+2\theta-\varepsilon,\gamma+2\theta-\varepsilon}_p(\M)\oplus\C.
\end{align*} 
The second statement holds as long as $ \gamma+2\theta-\varepsilon> \frac{n+1}{2}$.
\end{lem}
\begin{proof}
The first embeddings follow from \cite[Lemma~5.4]{CorSchSei02} and \cite[Lemma~3.6]{RoiSch15}. The second embeddings  were proved in \cite[Lemma~5.2]{RoiSch15}. The extra assumption $ \gamma+2\theta-\varepsilon> \frac{n+1}{2}$ is to guarantee the sums of the weighted Mellin Sobolev spaces with $\C$ are direct.
\end{proof}

\subsection{\bf Conic $BC^k$-spaces}\label{Subsection 2.2}
In this subsection, we will first give a brief introduction to the usual $BC^k$-spaces on 
$$
(\mathscr{M},g_{\mathscr{M}})\in \{(\tilde{\M},\tilde{g}), (\overline{\R}_+\times \B, dt^2+g_\B),(\Omega, dx_1^2+\cdots+dx_L^2)\},
$$ 
where $\Omega$ is an open subset of some Euclidean space $\R^L$. 

The  space $BC^k(\mathscr{M})$ is defined by
$$
BC^k(\mathscr{M}):=\{u\in C^k(\mathscr{M}):\|u\|_{k,\infty}:= \sum_{i=0}^k \|\nabla_{g_{\mathscr{M}}}^i u\|_\infty<\infty\}.
$$
We also set
$
BC^\infty(\mathscr{M}):=\bigcap\limits_{k\in\Nz} BC^k(\mathscr{M}).
$

Similarly, we can also define conic $BC^k$-spaces on $(\M,g)$ with respect to the conic metric $g$. 
More precisely, the space $\cC_g^k(\M)$ consists of all the $k-$times continuously differentiable functions $u$ such that 
$$
\|u\|_{k,\infty;c}:=\| (1-\omega)u \|_{k,\infty}+ \sum_{j+|\alpha|\leq k} \|(\partial_x)^j (\frac{1}{x}\partial_y)^\alpha (\omega u)\|_\infty<\infty.
$$
We set $\cC_g(\M):=\cC_g^0(\M)$ and
$$
\cC^\infty_g(\M):=\bigcap _{k\in\Nz} \cC^k_g(\M).
$$

\begin{remark}
In the above definitions of Mellin-Sobolev spaces and conic $BC^k$-spaces, those functions are the usual Bessel potential or $BC^k$-spaces outside $I\times \B$, and the conic metric $g=dx^2+ x^2 g_\B$ only plays a role near the conic singularities. Since we can always decompose a function $u$ into $u=\omega u + (1-\omega)u$.
Therefore, in the sequel, for a function in Mellin-Sobolev spaces or conic $BC^k$-spaces, we will only analyze it inside $I\times \B$.
\end{remark}

Recall that $S_\gamma u(t,y)=e^{(\gamma-\frac{n+1}{2})t} u(e^{-t},y)$. It is not a hard job to check that 
$$
 \|S_{\frac{n+1}{2}} (\omega u)\|_{k,\infty} \leq \|(\omega u)\|_{k,\infty;c} \leq  \|S_{\frac{n+1}{2}+k} (\omega u)\|_{k,\infty}
$$
for any $u\in \cC_g^k(\M)$. The rightmost term can be infinite.

This equality gives rise to the following two propositions.
\begin{prop}
\label{S2.2: Holder-iso}
$$
S_{\frac{n+1}{2}}\in \L (\cC_g^k(I\times \B), BC^k(\overline{\R}_+\times \B)),\quad k\in\Nz.
$$
\end{prop}

\begin{prop}
\label{S2.2: Sobolev-embedding}
Suppose that $s>k+\frac{n+1}{p}$ and $\gamma\geq \frac{n+1}{2}+k$. Then
$$
\cH^{s,\gamma}_p(\M)\hookrightarrow \cC_g^k(\M).
$$
\end{prop}
\begin{proof}
By definition, we have $\cH^{s,\gamma}_p(\M)\hookrightarrow \cH^{s,\frac{n+1}{2}+k}_p(\M)$. If $u\in \cH^{s,\frac{n+1}{2}+k}_p(\M)$, then it follows from the assumption $s>k+\frac{n+1}{p}$ and the conventional Sobolev-embedding theorem, cf. \cite[Theorem~14.2(ii)]{Ama13}, that
$$
\| S_{\frac{n+1}{2}+k} (\omega u)\|_{BC^k(\overline{\R}_+\times\B)} \leq C\| S_{\frac{n+1}{2}+k} (\omega u)\|_{H^s_p(\overline{\R}_+\times\B)}.
$$
Now applying the  conventional Sobolev embedding theorem once more to $(1-\omega)u$ will establish the desired result.
\end{proof}
\begin{remark}
It follows from \cite[Corollary 2.5]{RoiSch13} that if $s>\frac{n+1}{p}$, we actually have
the estimate
$
 |u(x,y)|\leq C x^{\gamma -\frac{n+1}{2}} \|u\|_{s,\gamma;p}
$.
\end{remark}

\subsection{\bf The Nemyskii operator in Mellin Sobolev spaces}\label{Subsection 2.3}
One of the crucial steps in showing the well-posedness and regularity of the harmonic map heat flow~\eqref{S1: HHF} is to obtain the necessary regularity for the operator
$$
[u\mapsto A_g(u)(\nabla u, \nabla u)].
$$
To this end, the objective of this subsection is to establish a Nemyskii operator theorem for Mellin Sobolev spaces defined in Section~\ref{Subsection 2.1}. 
However, in order not to cause too much distraction and to keep this article in a reasonable length, we will not state the optimal result, but only confine ourselves to the least necessary level.

For any Banach space $E$, one can  generalize the definitions of Mellin Sobolev and conic $BC^k$-spaces introduced in Sections~2.1~and~2.2 to $E$-valued spaces  without difficulty. 
We will denote these generalized spaces by $\cH^{s,\gamma}_p(\M,E)$ and $\cC_g^k(\M,E)$, respectively.

Suppose $X_i$ with $i=0,1,2$ are Banach spaces. We call
$$
\bullet: (v_1,v_2)\mapsto v_1\bullet v_2=v_0,\qquad v_i\in X_i,
$$
a Banach multiplication from $X_1\times X_2\to X_0$ if it satisfies
$$
\|v_0\|_{X_0} \leq c \|v_1\|_{X_1} \|v_2\|_{X_2}.
$$ 
\begin{theorem}
\label{S2.3: ptwise-mul}
Let $\bullet$ be a Banach multiplication from $X_1\times X_2\to X_0$. Then $[(v_1,v_2)\mapsto v_1\bullet v_2]$ is a continuous bilinear map from
\begin{itemize}
\item[(i)] $\cH^{s,\gamma}_p(\M,X_1)\times \cH^{s,\gamma^\prime}_p(\M,X_2)$ to $\cH^{s,\gamma+\gamma^\prime-\frac{n+1}{2}}_p(\M,X_0)$ if $s>\frac{n+1}{p}$. In particular, if in addition $\gamma\geq \frac{n+1}{2}$, $\cH^{s,\gamma}_p(\M)$ is a Banach algebra.
\item[(ii)] $\cH^{s,\gamma}_p(\M,X_1)\times \cC_g^k(\M,X_2)$ to $\cH^{s,\gamma}_p(\M,X_0)\times0)$ if $k\geq s$.
\end{itemize}
\end{theorem}
\begin{proof}
As we have pointed out in Section 2.2, only the part of a function in $I\times B$ will be considered.
For notional brevity, let $\vartheta=\gamma+\gamma^\prime-\frac{n+1}{2}$. For any $u\in \cH^{s,\gamma}_p(\M,X_1)$ and $v\in \cH^{s,\gamma^\prime}_p(\M,X_2)$, we have
\begin{align}
\label{S2.3: eq 2.2}
\notag \|u\bullet v\|_{s,\vartheta;p}&\leq c\|S_\vartheta(u \bullet v)(t,y)\|_{s,p}\\
\notag &=c\| e^{\vartheta-\frac{n+1}{2}} u(e^{-t},y)\bullet v(e^{-t},y)\|_{s,p}\\
\notag &=c\| e^{\gamma-\frac{n+1}{2}} u(e^{-t},y)\bullet e^{\gamma^\prime-\frac{n+1}{2}}v(e^{-t},y)\|_{s,p}\\
\notag &\leq c\| e^{\gamma-\frac{n+1}{2}} u(e^{-t},y)\|_{H^s_p(\bar{R}_+\times \B,X_1)} \| e^{\gamma^\prime-\frac{n+1}{2}} v(e^{-t},y)\|_{H^s_p(\bar{R}_+\times \B,X_2)}\\
&\leq c \|u\|_{\cH^{s,\gamma}_p(I\times \B,X_1)}   \|v\|_{\cH^{s,\gamma^\prime}_p(I\times \B,X_2)} .     
\end{align}
Step \eqref{S2.3: eq 2.2} follows from an easy generalization of the conventional pointwise multiplication theorem on compact manifolds to the half cylinder $\overline{\R}_+\times\B$. 

This proves (i). 
One can check directly via the definitions of Mellin Sobolev and conic $BC^k$-spaces that
Statement (ii) holds for $s\in\Nz$ and $s\leq k$. For non-integer $s$, the assertion follows from the interpolation theory and Lemma~\ref{S2.1: Sobolev-interpolation-complex}.
\end{proof}
 

In the rest of this section, we assume that $\cO\subset \R^L$ for some $L\in\N$ is open, $\mathscr{M}$ is a manifold, and $E$ is a Banach space. Let $\phi:\cO\to E$. Then the Nemyskii operator, or the so called substitution operator, $\Phi$ induced by $\phi$ is defined by
$$
\Phi: \cO^{\mathscr{M}}\to E^{\mathscr{M}}, \quad u\mapsto \phi(u).
$$
We use ${\rm Rng}(f)$ to denote the range of an $\R^L$-valued function $f$. For a set $X$ of $\R^L$-valued functions, similarly, we set
$$
{\rm Rng}(X)：= \bigcup_{f\in X} {\rm Rng}(f).
$$ 

\begin{prop}
\label{S2.3: Mapping prop of Nemyskii}
Suppose that  $p>n+1$ and $\gamma\geq \frac{n+1}{2}$.
Let $\phi\in BC^\nu(\cO,E)$ with $[s]\leq \nu \in\Nz $ and $X$ be an open subset of $\cH^{s,\gamma}_p(\M, \R^L )\oplus \R^L$ such that ${\rm Rng}(X)\subset \cO$. 
Then for any $\vartheta<\frac{n+1}{2}$
$$
\Phi(X)\subset \cH^{[s],\vartheta}_p(\M, E).
$$
\end{prop}
\begin{proof}
Let  $u\in X$. When $s\in [0,1)$, since $\|\Phi(u)\|_\infty<\infty$, for any $\vartheta<\frac{n+1}{2}$ we obviously have
$$
\Phi(u)\in  \cH^{0,\vartheta}_p(\M, E).
$$

For $s \geq 1$, note that 
$$
\partial^j \phi \in BC^{\nu-j}(\cO, \L^j(\R^L,E)),\quad j\leq \nu.
$$
The assertion now follows from the definition of Mellin Sobolev spaces, Theorem~\ref{S2.3: ptwise-mul} and the chain rule for higher order derivatives, that is, 
$$
\frac{\partial^l}{\partial z^l}\phi(u)=\sum_{k_1+2k_2+\cdots+l k_l=l} \frac{l!}{k_1! k_2!\cdots k_l!} \partial^k\phi(u) (\frac{\partial_z u}{1!})^{k_1}(\frac{\partial_z^2 u}{2!})^{k_2}\cdots (\frac{\partial_z^l u}{l!})^{k_l}
$$
where $z\in\{x,y_1,\cdots,y_n\}$, $l\leq [s]$ and $k=k_1+k_2+\cdots+k_l$.
\end{proof}

\begin{theorem}
\label{S2.3: Nemyskii thm}
Suppose that $k\in\Nz$, $s>\frac{n+1}{p}$, $p>n+1$ and $\gamma\geq \frac{n+1}{2}$.
Let $\phi\in BC^{\nu+k+1}(\cO,E)$ for some $[s]\leq \nu \in\Nz $ and $X$ be an open subset of $\cH^{s,\gamma}_p(\M, \R^L )\oplus \R^L$ such that ${\rm Rng}(X)\subset \cO$. 
Then for any $\vartheta<\frac{n+1}{2}$
$$
\Phi\in C^{k+1-}(X,\cH^{[s],\vartheta}_p(\M, E) ).
$$
\begin{proof}
For each $u\in X$, by the openness of $\cO$ and Proposition~\ref{S2.2: Sobolev-embedding}, we can find a bounded convex neighbourhood $U$ of $u$ in $\cH^{s,\gamma}_p(\M, \R^L)\oplus \R^L$ such that $U\subset X.$

(i) We first consider the case $k=0$. Pick $u,v\in U$. Then
$$
\Phi(u)-\Phi(v)=m(u,v)(u-v),
$$
where 
$$
m(u,v)=\int_0^1 \partial\phi(v+\tau(u-v))\, d\tau.
$$
Since $\partial \phi \in BC^\nu(\cO, \L(\R^L,E))$, by Proposition~\ref{S2.3: Mapping prop of Nemyskii} and its proof, one can find a constant $C=C(U)>0$ such that
$$
\|\partial\phi(w)\|_{\cH^{[s],\vartheta}_p(\M,\L(\R^L,E))}\leq C,\quad w\in U.
$$
So we can infer from $\gamma\geq \frac{n+1}{2}$, Proposition~\ref{S2.2: Sobolev-embedding} and Theorem~\ref{S2.3: ptwise-mul} that
$$
\|\Phi(u)-\Phi(v)\|_{\cH^{[s],\vartheta}_p(\M,E)}\leq C \|u-v\|_{\cH^{s,\gamma}_p(\M, \R^L)\oplus \R^L}.
$$
In particular, when $s<1$, we apply both Proposition~\ref{S2.2: Sobolev-embedding} and Theorem~\ref{S2.3: ptwise-mul}(ii) in the above inequality. When $s\geq 1$, applying only $\gamma\geq \frac{n+1}{2}$ and Theorem~\ref{S2.3: ptwise-mul}(i) is sufficient.
This proves the asserted statement for $k=0$.

(ii) Suppose that $k=1$. 
For any $h\in \cH^{s,\gamma}_p(\M, \R^L)\oplus \R^L$ with $\|h\|_{ \cH^{s,\gamma}_p(\M, \R^L)\oplus \R^L}$ small enough,
\begin{equation}
\label{S2.3: eq 2.3}
\Phi(u+h)-\Phi(u)-\partial\phi(u)h= \int_0^1 [\partial\phi(u+\tau h) -\partial \phi(u)]h\, d\tau
\end{equation}
holds point-wise on $\M$.

Because $\partial\phi\in BC^{1+\nu}(\cO, \L(\R^L,E))$, by Step (i), we know that the Nemyskii operator $B$ induced by $\partial\phi$ satisfies
$$
B\in C^{1-}(X , \cH^{[s],\vartheta}_p(\M, \L(\R^L,E)).
$$
Using \eqref{S2.3: eq 2.3}, one can obtain
\begin{align*}
&\|\Phi(u+h)-\Phi(u)-\partial\phi(u)h\|_{\cH^{[s],\vartheta}_p(\M,E)}\\
\leq & \int_0^1 \|[\partial\phi(u+\tau h) -\partial \phi(u)]h\|_{\cH^{[s],\vartheta}_p(\M,E)}\, d\tau\\
\leq &  \int_0^1 \|\partial\phi(u+\tau h) -\partial \phi(u) \|_{\cH^{[s],\vartheta}_p(\M, \L(\R^L,E)}\, d\tau \|h\|_{\cH^{s,\gamma}_p(\M, \R^L)\oplus \R^L}\\
\leq & C \|h\|_{\cH^{s,\gamma}_p(\M, \R^L)\oplus \R^L}^2
\end{align*}
via Proposition~\ref{S2.2: Sobolev-embedding} and Theorem~\ref{S2.3: ptwise-mul} as in Step (i). This implies that
$$
\Phi\in C^{2-}(X, \cH^{[s],\vartheta}_p(\M,E))
$$
with $\partial\Phi(u)h= \partial \phi(u)h$.

(iii) The general case $k>1$ follows now by an induction argument.
\end{proof}
\end{theorem}


\section{\bf The Laplace-Beltrami operator}
\label{Section 3}

In this section, we will state some known results about the Laplace-Beltrami operator on conic manifolds, which will serve as the theoretical basis for the analysis of the harmonic map heat flow in the next section.

The Laplace-Beltrami operator $\Delta_g$ induced by the conic metric $g$ is a second order differential operator, which near the conic singularities, i.e. inside $I\times \B$, can be written as
$$
\Delta_g = \frac{1}{x^2}  [ (x\partial_x)^2 +(n-1)(x\partial_x)  + \Delta_{\B} ],
$$
where $\Delta_{\B}$ is the Laplace-Beltrami operator on $\B$ with respect to $g_\B$.

This operator acts in a natural way on the scale of weighted Mellin-Sobolev spaces defined in Section~2:
$$
\Delta_g \in \L(\cH^{s+2,\gamma+2}_p(\M), \cH^{s,\gamma}_p(\M)),
$$
for all $s\geq 0$, $\gamma\in\R$ and $1<p<\infty$, as a bounded differential operator.

The conormal symbol of $\Delta_g$ is
$$
\sigma_\M(\Delta_g)(z):=z^2 - (n-1)z+\Delta_\B,\quad z\in \C.
$$
In particular, $\sigma_\M(\Delta_g)\in \cA(\C, \L(H^{s+2}_p(\B), H^s_p(\B)))$, where $\cA(\C,E)$ denotes the space of holomorphic $E$-valued functions on $\C$ for any Banach space $E$.
The definition of the conormal symbols is inspired by the fact that
$$
M(x\partial_x f)(z)=zM f(z),
$$
where the Mellin transform is defined by
$$
M f (z)=\int_0^\infty x^{z-1}f(x)\, dx,\qquad z\in \C.
$$
We would like to refer the reader to \cite{Les97} for more details of the Mellin transform and the conormal symbols.

\subsection{\bf $\mathscr{R}$-sectorial operators and $L_p$-maximal regularity}\label{Subsection 3.1}

Maximal regularity theory has proved itself a very important tool in the study of nonlinear parabolic equations. 
In combination with a fixed point argument and the implicit function theorem, the theory of maximal regularity can be used to establish wellposedness, regularity and stability of solutions.
In this subsection, we will introduce several important concepts and theorems in maximal regularity theory.  
The reader may refer to the treatises \cite{Ama95,  DenHiePru03, PruSim16} for a thorough survey of the theory. 

For $\theta\in (0,\pi]$, the open sector with angle $2\theta$ is denoted by
$$\Sigma_\theta:= \{\omega\in \mathbb{C}\setminus \{0\}: |\arg \omega|<\theta \}. $$
\begin{definition}
Let $X$ be a complex Banach space, and $\cA$ be a densely defined closed linear operator in $X$ with dense range. $\cA$ is called sectorial if $\Sigma_\theta \subset \rho(-\cA)$ for some $\theta>0$ and
$$ \sup\{\|\mu(\mu+\cA)^{-1}\|_{\L(X)} : \mu\in \Sigma_\theta \}<\infty. $$
The class of sectorial operators in $X$ is denoted by $\S(X)$. 
The spectral angle $\phi_\cA$ of $\cA$ is defined by
$$\phi_\cA:=\inf\{\phi:\, \Sigma_{\pi-\phi}\subset \rho(-\cA),\, \sup\limits_{\mu\in \Sigma_{\pi-\phi}} \|\mu(\mu+\cA)^{-1}\|_{\L(X)}<\infty. \}. $$ 
\end{definition}

\begin{definition}
Suppose that $\cA\in \S(X)$. Then $\cA$ is said to admit bounded imaginary powers if $\cA^{is}\in \L(X)$ for each $s\in\R$, and there exists some constant $C>0$ such that
$$\|\cA^{is}\|_{\L(X)} \leq C, \quad |s|\leq 1. $$
The class of such operators is denoted by $\mathcal{BIP}(X)$.
The power angle $\theta_\cA$ of $\cA$ is defined by
$$\theta_\cA:=\overline{\lim}_{|s|\to\infty} \frac{1}{|s|} \log \|\cA^{is}\|_{\L(X)} .$$
\end{definition}

\begin{definition}
Let $X$ and $Y$ be two Banach spaces. A family of operators $\mathcal{T}\in \L(X,Y)$ is called $\mathscr{R}$-bounded, if there is a constant $C>0$ and $p\in [1,\infty)$ such that for each $N\in\N$, $T_j\in \mathcal{T}$ and $x_j\in X$ and for all independent, symmetric, $\{-1,1\}$-valued random  variables $\varepsilon_j$ on a probability space $(\Omega,\mathcal{M},\mu)$ the inequality
$$ \|\sum\limits_{j=1}^N \varepsilon_j T_j x_j \|_{L_p(\Omega, Y)} \leq C \|\sum\limits_{j=1}^N \varepsilon_j  x_j \|_{L_p(\Omega, X)} $$
is valid. The smallest such $C$ is called the $\mathscr{R}$-bound of $\mathcal{T}$. We denote it by $\mathscr{R}(\mathcal{T})$.
\end{definition}

\begin{definition}
Suppose that $\cA\in \S(X)$. Then $\cA$ is called $\mathscr{R}$-sectorial if there exists some $\phi>0$ such that
$$\mathscr{R}_\cA(\phi):=\mathscr{R}\{\mu(\mu+\cA)^{-1}: \mu\in \Sigma_\phi \}<\infty. $$
The $\mathscr{R}$-angle $\phi^R_\cA$ is defined by
$$\phi^R_\cA:=\inf\{ \theta\in (0,\pi): \mathscr{R}_\cA(\pi-\theta) <\infty \}. $$
The class of $\mathscr{R}$-sectorial operators in $X$ is denoted by $\RS(X)$.
\end{definition}

\begin{definition}
A Banach space $X$ is said to belong to the class $\mathcal{HT}$ if the Hilbert Transform defined by
$$H(f)(t):=\lim\limits_{\epsilon\to 0} \int_{|s|>\epsilon} f(t-s)\frac{ds}{\pi s} ,\quad t\in\R,\, f\in C_0(\R,X) $$
can be extended to a bounded linear operator on $L_p(\R,X)$ for some $p\in (1,\infty)$.
\end{definition}
In particular,  Mellin-Sobolev spaces defined in Section~2.1 are of class $\mathcal{HT}$. 
Suppose that $X$ is a Banach space of class $\mathcal{HT}$. 
Then by \cite[formula~(2.15), Remark~3.2(1), Theorem~4.5]{DenHiePru03}, 
we obtain
the inclusions
\begin{equation}
\label{S3.1: all classes}
\mathcal{BIP}(X) \subset \RS(X) \subset \S(X),
\end{equation}
and the inequalities
\begin{equation}
\label{S3.1: all angles}
\theta_\cA \geq \phi^R_\cA \geq \phi_\cA. 
\end{equation}

\begin{definition}
\label{S3.1: Def-MR}
Assume that $X_1\overset{d}{\hookrightarrow}X_0$ is some densely embedded Banach couple.
Suppose that $\cA\in \S(X_0)$ with $\dom(\cA)=X_1$.
The Cauchy problem  
\begin{equation}
\label{S3.1: Cauchy problem}
\left\{\begin{aligned}
\partial_t u(t) +\cA u(t) &=f(t) &&t\geq 0\\
u(0)&=u_0  &&
\end{aligned}\right. 
\end{equation} 
has maximal $L_p$-regularity if for any 
$$(f,u_0)\in L_p(\R_+, X_0)\times (X_0,X_1)_{1-1/p,p} ,$$
\eqref{S3.1: Cauchy problem}
has a unique solution
$$u\in L_p(\R_+, X_1) \cap H^1_p(\R_+, X_0) .$$
We denote this by 
$$\cA\in \mathcal{MR}_p(X_0).$$
\end{definition}
Note that Definition~\ref{S3.1: Def-MR} basically tells us that an operator
$\cA\in \mathcal{MR}_p(X_0)$ iff the map 
\begin{equation}
\label{S3.1: MR-Lis}
(\partial_t +\cA, \gamma_0)\in\Lis(L_p(\R_+, X_1) \cap H^1_p(\R_+, X_0), L_p(\R_+, X_0)\times (X_0,X_1)_{1-1/p,p} ),
\end{equation}
where $\gamma_0$ is the temporal trace operator.

The class $\mathcal{MR}_p(X_0)$ is closely related to the class of $\mathscr{R}$-sectorial operators.
\begin{prop}
\label{S3.1: Prop-MR-RS}
\cite[Theorem~4.4]{DenHiePru03}
Assume that $X_1\overset{d}{\hookrightarrow}X_0$ is some densely embedded Banach couple.
Suppose that $\cA\in \S(X_0)$ with $\dom(\cA)=X_1$, and $X_0$ belongs to the class $\mathcal{HT}$. Then the Cauchy problem \eqref{S3.1: Cauchy problem} has maximal $L_p$-regularity for all $1<p<\infty$ iff 
$$
\cA\in \RS(X_0)
$$
with $\phi^R_\cA<\pi/2$.
\end{prop}

Maximal regularity theory is a powerful tool in the
theory of nonlinear parabolic equations. 
To illustrate its power, let us consider the following
abstract evolution equation 
\begin{equation}\label{S3.1: evolution eq}
\begin{cases}
 \partial_t u +\cA u =F(u), &t > 0;\\
 u(0)=u_0,&
\end{cases} 
\end{equation}
in $X_0$. We have the following existence and uniqueness result for
equation~\eqref{S3.1: evolution eq}. 
\begin{theorem}
\label{S3.1: Thm-MR-SL}
\cite[Theorem~2.1]{CleLi93}
Let $1<p<\infty$ and $X_1\overset{d}{\hookrightarrow}X_0$ be a densely
embedded pair of Banach spaces. 
Assume that $\cA\in \mathcal{MR}_p(X_0)$ with $\dom(\cA)=X_1$.
Setting $X_{1-1/p}:=(X_0,X_1)_{1-1/p,p}$,
suppose that $U\subset X_{1-1/p}$ is open and  that $F$ satisfies 
$$
 F\in C^{1-}\bigl(U,  X_0\bigr).
$$
Then for every $u_0\in U$, there exist $T=T(u_0)>0$ and a unique
solution of \eqref{S3.1: evolution eq} on $J=[0,T)$ with
$$
 u\in L_p(J, X_1)\cap H^1_p(J, X_0).
$$
\end{theorem}

\subsection{\bf Closed extensions of the Laplace-Beltrami operator on conic manifolds}\label{Subsection 3.2}

In this subsection, we will quote some well-established results on the closed extensions of the conic Laplace-Beltrami operator. More details of these results can be found in \cite{RoiSch14, RoiSch15,SchSei05}.

If we consider $\Delta_g$ as an unbounded operator on $\cH^{s,\gamma}_p(\M)$ with domain $C^\infty_c(\M)$,  denote its closure by $\Delta_{g, \min}=\Delta_{g,s,\min}^\gamma$, and its maximal closed extension by $\Delta_{g,\max}=\Delta_{g,s,\max}^\gamma$, where
$$
\dom(\Delta_{g,\max})=\{u\in \cH^{s,\gamma}_p(\M): \Delta_g u\in \cH^{s,\gamma}_p(\M)\}.
$$
We have
$$
\dom (\Delta_{g,\min})= \dom(\Delta_{g,\max}) \cap \bigcap_{\varepsilon>0} \cH^{s+2,\gamma+2-\varepsilon}_p(\M).
$$
In particular, $\dom (\Delta_{g,\min})=\cH^{s+2,\gamma+2}_p(\M)$ iff $\sigma_\M(\Delta_g)(z)$ is invertible  for all $z$ satisfying $\Rp z=\frac{n+1}{2}-\gamma-2$.
The reader may refer to \cite[Proposition~5.1]{SchSei05} for the details of this result.

We denote by $0=\lambda_0>\lambda_1>\cdots$ the distinct eigenvalues of $\Delta_\B$ and by $E_0, E_1, \cdots$ the
corresponding eigenspaces.
Then the non-bijectivity points of $\sigma_\M(\Delta_g)$ are exactly
$$
q^{\pm}_j= \frac{n-1}{2} \pm \sqrt{(\frac{n-1}{2})^2 -\lambda_j},\quad j\in\Nz.
$$
Note that $q^+_j= (n-1) -q^-_j$ and $q^-_0=0$. From the discussion in the previous paragraph, in case $q^{\pm}_j\neq \frac{n+1}{2}-\gamma-2$ for all $j\in \Nz$, it holds that
$$
\dom (\Delta_{g,\min})=\cH^{s+2,\gamma+2}_p(\M).
$$
For $q^{\pm}_j$ with $j\neq 0$, we define the function spaces
$$
\cE_{q^{\pm}_j}=\omega x^{-q^{\pm}_j}\otimes E_j = \{\omega(x)x^{-q^{\pm}_j} e_j(y): e_j \in E_j \}.
$$
Recall that $x$ is the singularity (boundary) defining function. When $j=0$, we put
\begin{equation*}
\cE_{q^{\pm}_0}=
\begin{cases}
\omega x^{q^{\pm}_0}\otimes E_0 \qquad &n>1;\\
\omega\otimes E_0 + \omega \log x\otimes E_0  &n=1.
\end{cases}
\end{equation*}
We will also introduce the set $I_\gamma$ defined by
$$
I_\gamma=\{q^{\pm}_j: j\in\Nz \}\cap (\frac{n+1}{2}-\gamma-2, \frac{n+1}{2}-\gamma).
$$
As a conclusion from \cite[Theorem~3.6]{GilKra06} and \cite[Proposition~5.1]{SchSei05}, we have the following proposition concerning the maximal domain of $\Delta_g$.
\begin{prop}
\label{S3.2: Prop-Delta-max-min}
Suppose that $q^{\pm}_j\neq \frac{n+1}{2}-\gamma-2$ for all $j\in \Nz$. Then
$$
\dom (\Delta_{g,\max})=\cH^{s+2,\gamma+2}_p(\M)\oplus \bigoplus_{q^{\pm}_j\in I_\gamma} \cE_{q^{\pm}_j}.
$$
\end{prop}

Given a subspace $\underline{\cE}_{q^{\pm}_j}$ of $\cE_{q^{\pm}_j}$, we associate with it a subspace $\underline{\cE}_{q^{\pm}_j}^\perp$ according to the following rules:
\begin{itemize}
\item[(i)] Suppose that either $q^{\pm}_j\neq 0$ or $n>1$. If $\underline{\cE}_{q^{\pm}_j}=\omega x^{-q^{\pm}_j}\otimes \underline{E}_j$ for some subspace $\underline{E}_j\subset E_j$, then we define
$$
\underline{\cE}_{q^{\pm}_j}^\perp=\omega x^{-q^{\mp}_j}\otimes \underline{E}_j^\perp,
$$
where $\underline{E}_j^\perp$ is the orthogonal complement of $\underline{E}_j$ in $E_j$ with respect to the $L_2(\B)$-inner product.
\item[(ii)] When $q^{\pm}_0= 0$ and $n=1$, define $\underline{\cE}_0^\perp=\{0\}$ if $\underline{\cE}_0=\cE_0$,
$\underline{\cE}_0^\perp=\cE_0$ if $\underline{\cE}_0=\{0\}$, $\underline{\cE}_0^\perp=\underline{\cE}_0$ if $\underline{\cE}_0=\omega \otimes E_0$.
\end{itemize}
Note that $\underline{\cE}_{q^{\pm}_j}^\perp$ is a subspace of $\cE_{q^{\mp}_j}$.

\begin{definition}
\label{S3.2: def-ext-Delta}
We define the extension $\underline{\Delta}_g=\underline{\Delta}_{g,s}^\gamma$ of $\Delta_g$ with the domain
$$
\dom(\underline{\Delta}_g)=\cH^{s+2,\gamma+2}_p(\M) \oplus \bigoplus_{q^{\pm}_j\in I_\gamma} \underline{\cE}_{q^{\pm}_j} 
$$
with $\underline{\cE}_{q^{\pm}_j}$ chosen as follows
\begin{itemize}
\item[(i)] if $q^{\pm}_j \in I_\gamma \cap I_{-\gamma}$, then $\underline{\cE}_{q^{\pm}_j}^\perp=\underline{\cE}_{q^{\mp}_j}$;
\item[(ii)] if $\gamma\geq 0$ and $q^{\pm}_j \in I_\gamma \setminus I_{-\gamma}$, then $\underline{\cE}_{q^{\pm}_j}^\perp= \cE_{q^{\pm}_j}$;
\item[(iii)] if $\gamma\leq 0$ and $q^{\pm}_j \in I_\gamma \setminus I_{-\gamma}$, then
$\underline{\cE}_{q^{\pm}_j}^\perp=\{0\}$.
\end{itemize}
\end{definition}

In the sequel, we will confine us to the extension $\underline{\Delta}_g$. This is a closed extension of $\Delta_g$. Indeed by Proposition~\ref{S3.2: Prop-Delta-max-min} an extension of $\Delta_g$ is closed iff its domain is of the form 
$$
\cH^{s+2,\gamma+2}_p(\M)\oplus \cE, 
$$
where $\cE$ is a subspace of $\bigoplus\limits_{q^{\pm}_j\in I_\gamma} \cE_{q^{\pm}_j}$,
as long as $q^{\pm}_j\neq \frac{n+1}{2}-\gamma-2$ for all $j\in \Nz$.

With all these preparations, we are ready to state the following result on the Laplace-Beltrami operator based on \cite[Theorem~3.11]{RoiSch15}.
\begin{theorem}
\label{S3.2: Delta-MR}
Assume $s\geq 0$, $|\gamma|<\frac{1}{2}{\rm dim}(\M)$. Then the extension $\underline{\Delta}_g$ defined in Definition~\ref{S3.2: def-ext-Delta} satisfies 
$$
c-\underline{\Delta}_g \in \mathcal{MR}_p(\cH^{s,\gamma}_p(\M))
$$
for sufficiently large $c>0$.
\end{theorem}
\begin{proof}
Under the given conditions, \cite[Theorem~3.11]{RoiSch15} states that
$$
c-\underline{\Delta}_g \in \mathcal{BIP}(\cH^{s,\gamma}_p(\M))
$$
with power angle $\theta_{c-\underline{\Delta}_g}<\pi/2$,
for $c>0$ large enough. By \eqref{S3.1: all classes}, the $\mathscr{R}$-angle of $c-\underline{\Delta}_g$ is smaller than $\pi/2$.
Now the assertion follows from Proposition~\ref{S3.1: Prop-MR-RS}.
\end{proof}



\section{\bf Short time existence of the harmonic map heat flow on conic manifolds}
\label{Section 4}

Throughout, suppose that $u$ is a harmonic map heat flow from an $n+1$-dimensional conic manifold $(\M,g)$ to an $m$-dimensional smooth closed manifold $(\sN,h)$.

Recall that the harmonic map heat flow for $u:\M_T\to \sN$ is
\begin{equation}
\label{S4: HHF}
\left\{\begin{aligned}
\partial_t u - \Delta_g u &=A_g(u)(\nabla u, \nabla u)  &&\text{on}&&\M_T;\\
u(0)&=u_0  &&\text{on}&&\M,&&
\end{aligned}\right.
\end{equation}
where  the initial condition $u_0:\M\to \sN$, and 
$\M_T:=(0,T)\times\M$ for some $T>0$.

By the Nash embedding theorem, we may assume that $(\sN,h)$ is isometrically embedded into a Euclidean space $\R^L$. 
Since $\sN$ is smooth and closed, we can always find a $a$-tubular neighbourhood 
$$
\T_a:=\{z\in \R^L: {\rm dist}(z,\sN)<a\}
$$ 
for sufficiently small $a>0$ such that within which the nearest point projection $\pi_\sN$, that is the map satisfying
$$
|z-\pi_\sN(z)|={\rm dist}(z,\sN)\qquad z\in \T_a,
$$ 
is a smooth function onto $\sN$. Without confusion, we also use $\pi_\sN: \R^L \to \R^L$ to denote a smooth extension of the nearest point projection $\pi_\sN$ from $\T_a$ to $\sN$.

We will first consider $u$ as a function from $\M$ to $\R^L$. With this relaxation, let us first study an alternative problem to \eqref{S4: HHF}
\begin{equation}
\label{S4: HHF-2}
\left\{\begin{aligned}
\partial_t u - \Delta_g u &=\phi(u)(\nabla u, \nabla u)_{g^*}  &&\text{on}&&\M_T;\\
u(0)&=u_0  &&\text{on}&&\M,&&
\end{aligned}\right.
\end{equation}
with initial value $u_0:\M\to \sN$, where $\phi=\partial^2 \pi_\sN$ and 
$$\phi(u)(\nabla u,\nabla u)_{g^*} =g^{\alpha\beta}\phi(u)\big(\frac{\partial u}{\partial x_\alpha}, \frac{\partial u}{\partial x_\alpha}\big).$$

In order to apply Theorems~\ref{S3.1: Thm-MR-SL} and \ref{S3.2: Delta-MR} to \eqref{S4: HHF-2}, we will impose the following conditions.

\pagebreak

\begin{mdframed}
\begin{itemize}
\item[(A)] When $n=1,2$, we assume that 
$$
4/p<\gamma<2,\quad  p>6,
$$
and
$
\gamma<\sqrt{4-n-\lambda_1}-1.
$
\end{itemize}
\end{mdframed}
or 
\vspace{.5em}
\begin{mdframed}
\begin{itemize}
\item[(B0)] $0<\gamma<\frac{n+1}{2}$, $p>n+3$ and $\gamma+2-\frac{4}{p}>\frac{n+1}{2}$.
\item[(B1)] 
When $n=1$ and $\lambda_1<-1$, we choose $\gamma>0$ so small  that
$$
-\sqrt{-\lambda_1}<-1-\gamma.
$$
\item[(B2)] When $n=2$ and $\lambda_1<-2$, $\gamma\in (1/2,3/2)$ is close enough to $1/2$ such that
$$
q^-_1<-\frac{1}{2}-\gamma<-1.
$$
\item[(B3)] When $n\geq 3$, 
$$
q_1^- <\frac{n+1}{2}-\gamma-2.
$$
\end{itemize}
\end{mdframed}
\begin{rmk}
(i) Although (B0)-(B2) have some overlapping with (A), they can render us some flexibility in choosing the parameter $p$ when the first eigenvalue $\lambda_1$ of $\B$ satisfies the given condition.

(ii) The functions of Assumptions~(B0)-(B3) is to fulfil the conditions in Theorem~\ref{S3.2: Delta-MR} and to obtain enough regularity for the map 
$$
[u\mapsto \phi(u)(\nabla u, \nabla u)_{g^*} ].
$$
The latter only uses (B0) and will be clarified right after this remark. We will give some explanation for the former here. 

(B0) ensures that $|\gamma|<\dim(\M)/2$. We will show that, for any $s\geq 0$, 
$$
{\rm dom}\underline{\Delta}_{g,s}^\gamma= \cH^{s+2,\gamma+2}_p(\M)\oplus \C
$$
as long as (B0) and (B$j$) are met when $j=\min\{3,n \}$ for $j\in\N$. 

{\it Case 1: }
Suppose that $n\geq 3$. 
The third condition in (B0) implies 
$$
\gamma+2>\frac{n+1}{2} \Longleftrightarrow \frac{n+1}{2}-\gamma-2<0,
$$
and thus $q_0^+=n-1\geq 2>\frac{n+1}{2}-\gamma$. Combining with (B3), this yields
$$
I_\gamma=\{0\}=\{q^-_0\}. 
$$
On the other hand, $\frac{n+1}{2}+\gamma-2>0$ implies that
$$
q^-_0\in I_\gamma\setminus I_{-\gamma}.
$$
By Definition~\ref{S3.2: def-ext-Delta}(ii), we conclude
$$
\underline{\cE}_{q^-_0}=\cE_{q^-_0}=\C.
$$

{\it Case 2: } 
Suppose that $n=2$. Due to (B2), 
$$
\{q^-_0\}\in I_\gamma\setminus I_{-\gamma}
$$
and 
$
q^+_0=1>\frac{3}{2}-\gamma.
$
Hence Definition~\ref{S3.2: def-ext-Delta}(ii) again yields
$$
\underline{\cE}_{q^-_0}=\cE_{q^-_0}=\C.
$$

{\it Case 3: }
Suppose that $n=1$. Note that in this case $q^\pm_0=0$ and $q^\pm_1=\pm \sqrt{-\lambda_1}$. 
(B0) implies
$$
q^\pm_0\in I_\gamma \cap I_{-\gamma}.
$$
At the same time,  (B1) guarantees that
$$
q^-_1<-1-\gamma<1-\gamma<1<q^+_1,
$$
and thus 
$$
I_\gamma=\{0\}=\{q^\pm_0\}.
$$
Now it follows from Definition~\ref{S3.2: def-ext-Delta}(i) that
$$
\underline{\cE}^\perp_0=\underline{\cE}_0=\C.
$$

All the discussions in (ii) imply that
$$
{\rm dom}\underline{\Delta}_{g,s}^\gamma= \cH^{s+2,\gamma+2}_p(\M)\oplus \C,\quad s\geq 0.
$$
Therefore, Theorem~\ref{S3.2: Delta-MR} will apply as long as (B0) and (B$j$) are assumed when $j=\min\{3,n \}$. 

(iii) As we can see from (ii), the extra conditions on $\lambda_1$ in (B1) and  (B2) are coming from the need to avoid containing $q_1^-$ and $q^+_0$ in the interval $(\frac{n+1}{2}-\gamma-2, \frac{n+1}{2}-\gamma)$ such that
$$
{\rm dom}\underline{\Delta}_{g,s}^\gamma= \cH^{s+2,\gamma+2}_p(\M)\oplus \C.
$$
If any of $q_1^-$ and $q^+_0$ is contained in $(\frac{n+1}{2}-\gamma-2, \frac{n+1}{2}-\gamma)$, 
the asymptotic behavior of ${\rm dom}\underline{\Delta}_{g,s}^\gamma$ will become much more complicated, as indicated by Definition~\ref{S3.2: def-ext-Delta}. 
This in turn will create an essential difficulty in computing 
$$
(\cH^{s,\gamma}_p(\M),  {\rm dom}\underline{\Delta}_{g,s}^\gamma)_{1-1/p,p},
$$
and as we learn from Definition~\ref{S3.1: Def-MR}, the above interpolation space is the space of initial data of \eqref{S4: HHF}.

However, when $n\geq 3$, we have $q^+_0=n-1\geq 2$. This provides us with enough room to avoid both $q_1^-$ and $q^+_0$ in 
$(\frac{n+1}{2}-\gamma-2, \frac{n+1}{2}-\gamma)$.
\end{rmk}


In the sequel, we will always assume that  (B0) and (B$j$) hold when $j=\min\{3,n \}$.


For notational brevity, for any $\vartheta\in\R$, we put 
$$
\vartheta_{\varepsilon,p}=\vartheta+2-2/p-\varepsilon.
$$

One can compute inside $I\times \B$ that
$$
(\nabla u, \nabla v)_{g^*}  = \partial_x u \partial_x v +\frac{1}{x^2} (\nabla_{g_\B}u, \nabla_{g_\B} v)_{g^*_\B},
$$
where $g^*_\B$ is the induced cotangent metric by $g_\B$. 
Since for any $s\geq 0$ and $\vartheta\in\R$
$$
[u\mapsto \partial_x u]\in \L(\cH^{s+1,\vartheta+1}_p(\M), \cH^{s,\vartheta}_p(\M))
$$
and 
$$
[u\mapsto \frac{1}{x}\nabla_{g_\B} u]\in \L(\cH^{s+1,\vartheta+1}_p(\M), \cH^{s,\vartheta}_p(\M, T^*\B)),
$$
we can obtain the following lemma in virtue of Theorem~\ref{S2.3: ptwise-mul}(i).
\begin{lem}
\label{S4: nabla-pt-mul}
If (B0) is assumed and let $\varepsilon>0$ so small that
\begin{equation}
\label{S4: ASP on epsilon-1}
1-\frac{n+3}{p}-\varepsilon>0,
\end{equation}
then for any $s\geq 0$ 
$$
[(u,v)\mapsto (\nabla u, \nabla v)_{g^*} ]
$$
is a continuous bilinear map from 
$$
(\cH^{\sep, \gep}_p(\M)\oplus\C)\times (\cH^{\sep, \gep}_p(\M)\oplus\C)\to 
\cH^{\sep-1, 2\gep-2-\frac{n+1}{2}}_p(\M).
$$ 
\end{lem}

Recall that $\phi=\partial^2 \pi_\sN$ and $\pi_\sN\in C^\infty(\T_a,\sN)$. By possibly further shrinking $a$, we have
$$
\phi\in BC^\infty(\T_a, \L^2(\R^L,\R^L)).
$$
Assume that the conditions of Lemma~\ref{S4: nabla-pt-mul} are satisfied.
In view of (B0), we can find $\varepsilon>0$ so small that both \eqref{S4: ASP on epsilon-1} and
\begin{equation}
\label{S4: ASP on epsilon-2}
\gamma_{\varepsilon,p}\geq \frac{n+1}{2}
\end{equation}
are fulfilled. 

Let $\p\in\sN\subset \R^L$.
For every $v=w+\p$ with $w\in \cH^{\sep, \gep}_p(\M,\sN-p)$ satisfying \eqref{S4: ASP on epsilon-1} and \eqref{S4: ASP on epsilon-2}, by Proposition~\ref{S2.2: Sobolev-embedding}, we can find a bounded neighbour $U=U(v)$ of $v$ in $\cH^{\sep, \gep}_p(\M,\R^L)\oplus \R^L$ such that
$$
{\rm Rng}(U)\subset \T_a.
$$
Then we can apply Theorem~\ref{S2.3: Nemyskii thm} and infer that the Nemyskii operator $\Phi$ induced by $\phi$ satisfies
$$
\Phi\in C^\infty(U, \cH^{[\sep],\vartheta}(\M,\L^2(\R^L,\R^L))
$$
for any $\vartheta<\frac{n+1}{2}$.

Making use of Theorem~\ref{S2.3: ptwise-mul}(i) and Lemma~\ref{S4: nabla-pt-mul}, direct computations show that
\begin{equation}
\label{S4: A-reg-1}
[u\mapsto \phi(u)(\nabla u , \nabla u)_{g^*} ]\in C^\infty(U, \cH^{\sep-1,2\gep-2+\vartheta-n-1}(\M,\R^L)).
\end{equation}
In view of the third condition in (B0) and by choosing $\varepsilon>0$ sufficiently small and $\vartheta$ close enough to $\frac{n+1}{2}$, we can always have
\begin{equation}
\label{S4: A-weight}
2\gep-2+\vartheta-n-1>\gamma.
\end{equation}

Now we put 
$$
E_0^s=\cH^{s,\gamma}_p(\M,\R^L),\qquad E_1^s=\cH^{s+2,\gamma+2}_p(\M,\R^L)\oplus\R^L,
$$
and 
$$
E_{1-1/p}^s=(\cH^{s,\gamma}_p(\M,\R^L), \cH^{s+2,\gamma+2}_p(\M,\R^L)\oplus \R^L)_{1-1/p,p}.
$$
Due to Lemma~\ref{S2.1: Sobolev-interpolation}, we have the embedding
$$
E_{1-1/p}^s\hookrightarrow  \cH^{\sep,\gep}(\M,\R^L)\oplus \R^L.
$$
Denote by $\iota:E_{1-1/p}^s\to   \cH^{\sep,\gep}(\M,\R^L)\oplus \R^L $ the inclusion map. Let $\cO=\cO(v):=\iota^{-1}(U)$ be an open subset in $E_{1-1/p}^s$. 
We thus infer from \eqref{S4: A-reg-1} and \eqref{S4: A-weight} that
\begin{equation}
\label{S4: A-reg}
[u\mapsto \phi(u)(\nabla u , \nabla u)_{g^*} ]\in C^\infty(\cO, E_0^s).
\end{equation}
Careful readers may have noticed that the regularity in \eqref{S4: A-reg}  is much stronger than what is asked by Theorem~\ref{S3.1: Thm-MR-SL}, but later we will find out that \eqref{S4: A-reg} is indeed necessary for 
obtaining a smooth solution of the harmonic map heat flow.

Now we can use \eqref{S4: A-reg} and apply Theorem~\ref{S3.1: Thm-MR-SL} to prove the local well-posedness of \eqref{S4: HHF-2}.
\begin{theorem}
\label{S4: Wellposed-HHF2}
Assume that $\gamma$ and $p$ satisfy (B0) and (B3) when $n\geq 3$ or (B0) and (B2) when $n=2$ or (B0)-(B1) when $n=1$,
and the initial condition in \eqref{S4: HHF-2} satisfies 
$$
u_0=w_0+\p\quad \text{with}\quad w_0\in \cH^{2-\frac{2}{p}+\delta, \gamma+2-\frac{2}{p}+\delta}(\M,\sN-\p) \quad \text{and} \quad \p\in\sN
$$ 
for an arbitrary positive constant $\delta$. Then \eqref{S4: HHF-2} with initial condition $u_0$ has a unique solution
$$
u\in L_p(J_T, \cH^{2,\gamma+2}_p(\M,\R^L)\oplus \R^L)\cap H^1_p(J_T,\cH^{0,\gamma}_p(\M,\R^L) )
$$
on $J_T=[0,T)$ for some $T>0$. Moreover, for any $\varepsilon>0$
$$
u\in BC(J_T, \cH^{2-\frac{2}{p}-\varepsilon, \gamma+2-\frac{2}{p}-\varepsilon}(\M,\R^L)\oplus\R^L).
$$
\end{theorem}
\begin{proof}
We only need to show the last inclusion. It follows from \cite[Chapter III: Theorem~4.10.2]{Ama95} that
$$
u\in BC(J_T, (\cH^{0,\gamma}_p(\M,\R^L), \cH^{2,\gamma+2}_p(\M,\R^L)\oplus \R^L)_{1-1/p,p}).
$$
Now the inclusion is a direct result of Lemma~\ref{S2.1: Sobolev-interpolation}.
\end{proof}
\begin{remark}
\label{S4: HHF2 higher regularity}
Because the argument leading to Theorem~\ref{S4: Wellposed-HHF2} is independent of the choice $s\geq 0$, once (B0) and (B$j$) are assumed with $j=\min\{3,n\}$, for every $s\geq 0$ and
$$
u_0=w_0+\p\quad \text{with}\quad w_0\in \cH^{s+2-\frac{2}{p}+\delta, \gamma+2-\frac{2}{p}+\delta}(\M,\sN-\p) \quad \text{and} \quad \p\in\sN,
$$ 
\eqref{S4: HHF-2} with initial condition $u_0$ has a unique solution
$$
u\in L_p(J_T, \cH^{s+2,\gamma+2}_p(\M,\R^L)\oplus \R^L)\cap H^1_p(J_T,\cH^{s,\gamma}_p(\M,\R^L) )
$$
on $J_T=[0,T)$ for some $T=T(s)>0$.
In addition,
$$
u\in BC(J_T, \cH^{s+2-\frac{2}{p}-\varepsilon, \gamma+2-\frac{2}{p}-\varepsilon}(\M,\R^L)\oplus\R^L).
$$
\end{remark}


In the rest of this section, our objective is to show that the solution $u$ obtained in Theorem~\ref{S4: Wellposed-HHF2} actually solves \eqref{S4: HHF}. To this end, we put
$$
\rho(t)=|\pi_\sN (u(t)) - u(t) |^2.
$$
Direct computations show that
$$
\partial_t \rho= 2 (\pi_\sN(u)-u, \partial \pi_\sN(u) \partial_t u - \partial_t u),
$$
where $(\cdot,\cdot)$ is the standard inner product in $\R^L$, and
\begin{align*}
\Delta_g \rho= & 2 |\nabla (\pi_\sN(u)-u)|_g^2 + 2(\pi_\sN(u)-u, \partial\pi_\sN(u)\Delta_g u - \Delta_g u )\\
&+ 2 (\pi_\sN(u)-u, \partial^2 \pi_\sN(u)(\nabla u, \nabla u)_{g^*} ).
\end{align*}

Suppose that $z\in \T_a$ with $\pi_\sN(z)=z_0\in \sN$. Then there exist some $\nu\in (T_{z_0}\sN)^\perp$ and $s_z\in (-a,a)$ such that
$z=z_0+s_z \nu$.
Since $\pi_\sN(z_0+s\nu)=z_0$ for all $s\in (-a,a)$, one can derive that $\partial \pi_\sN(z)\nu=0$.

In view of the above consideration, we can infer from the expressions of $\partial_t \rho$ and $\Delta_g  \rho$ that
$$
\partial_t \rho - \Delta_g \rho = -2|\nabla(\pi_\sN(u)-u)|_g^2.
$$
Since $\M$ only has isolated point singularities, multiplying the above equality by $\rho$ and integrating over $\M$ yields
$$
 \partial_t\int_\M \rho^2\, dv_g + 2 \int_\M |\nabla \rho|_{g}^2 \, dv_g \leq 0.
$$
Since $\rho(0)\equiv 0$, the above integral inequality implies 
$$
\rho(t)\equiv 0, \quad t\in J_T.
$$
We have thus proved that 
$$
\pi_\sN(u(t))=u(t),\quad t\in J_T.
$$

Now it follows from \cite[Lemma~3.2]{Mos05} that 
$$
\partial^2\pi_\sN(u)(\nabla u, \nabla u)_{g^*} =A_g(u)(\nabla u, \nabla u).
$$
Therefore $u$ is indeed a solution to \eqref{S4: HHF}.


From now on, we consider the case that Condition (A) holds.
In this situation, we build a conic manifold $(\widehat{\M},\widehat{g})$ up on the manifold $\tilde{\M}\times S^l$ with $l=3-n$.

More precisely, let $\widehat{\B}:=\B\times S^l$ equipped with the product metric $g_{\widehat{\B}}=g_\B + g_{S^l}$, where $g_{S^l}$ is the standard metric on $S^l$.
Near the conic singularities, we equip $(0,1]\times (\B\times S^l)$ with the metric
$$
\widehat{g}=dx^2 +x^2 g_{\widehat{\B}},
$$
and, outside $(0,1]\times (\B\times S^l)$, $\widehat{\M}$ is endowed with the product metric $\tilde{g} + g_{S^l}$.

For every $u_0=w_0+\p$ with $w_0\in \cH^{2-\frac{2}{p}+\delta,\gamma+2-\frac{2}{p}+\delta}(\M,\sN-\p)$, we can define a new initial datum $\widehat{u}_0(z,\theta)=u_0(z)$, where $\theta$ is the coordinates for  $S^l$ and $z\in\M$. It is not hard to see that $\widehat{u}_0 = \widehat{w}_0 + \p$ with some $\widehat{w}_0\in  \cH^{2-\frac{2}{p}+\delta,\gamma+2-\frac{2}{p}+\delta}(\widehat{\M},\sN-\p)$. 

Note that the base $\widehat{\B}$ of the conic manifold $(\widehat{\M},\widehat{g})$ is of dimension $3$, and  Condition (A) has guaranteed that the new manifold  
$(\widehat{\M},\widehat{g})$ satisfies (B0) and  (B3). 
Therefore Theorem~\ref{S4: Wellposed-HHF2} and the above discussion imply that \eqref{S4: HHF} from $\widehat{\M} \to \sN$ with initial condition $\widehat{u}_0$ has a unique solution
$$
\widehat{u} \in L_p(J_T, \cH^{2,\gamma+2}_p(\widehat{\M},\R^L)\oplus \R^L)\cap H^1_p(J_T,\cH^{0,\gamma}_p(\widehat{\M},\sN) ).
$$
Suppose that $R$ is an arbitrary rotation matrix in $\R^{l+1}$, i.e. $R\in SO(l+1)$. Define
$$
v_{R}(z,\theta):=v(z,R\theta); \ {\rm{and}}\ \ u_R(z,\theta, t):=u(z, R\theta, t) \ {\rm{for}}\ (z,\theta)\in \M\times S^l, \ t\ge 0.
$$
Since the Dirichlet energy~\eqref{S1: Energy} is invariant under $R\in SO(l+1)$, i.e. $E(v)=E(v_R) $ and \eqref{S4: HHF} is the negative $L_2$-gradient flow of \eqref{S1: Energy},
$\widehat{u}_R$ too solves \eqref{S1: HHF} for any $R\in SO(l+1)$.

In view of $\widehat{u}_0=(\widehat{u}_{0})_R$ for any $R\in SO(l+1)$, it follows from the uniqueness of solution to \eqref{S1: HHF} that 
$\widehat{u}_R(z,\theta, t)=\widehat{u}(z,\theta, t)$ on $\widehat{\M}\times [0,T)$ for any $R\in SO(l+1)$.
Hence $\widehat{u}$ is indeed independent of the spherical variable $\theta\in S^l$.

Now we can just look at a hypersurface $\{\theta=\theta_0\}$ of  $\widehat{\M}$, the restriction of $\widehat{u}$ on this hypersurface generates a solution
$$
u \in L_p(J_T, \cH^{2,\gamma+2}_p(\M,\R^L)\oplus \R^L)\cap H^1_p(J_T,\cH^{0,\gamma}_p(\M,\sN) )
$$
to \eqref{S4: HHF} from $\M\to \sN$ with initial value $u_0$.
To prove the uniqueness of  solutions, we assume that $v$ is another solution to \eqref{S4: HHF} from $\M\to \sN$ with initial value $u_0$ that is different from $u$. 
Now $\widehat{v}(\cdot, \theta)= v(\cdot )$ also solves \eqref{S4: HHF} from $\widehat{\M}\to \sN$ with initial value $\widehat{u}_0$, which is different from
$\widehat{u}$. This violates the uniqueness of solutions to \eqref{S1: HHF} in the case $n=3$.

To sum up, we are ready to state the first  main theorem of this article.
\begin{theorem}
\label{S4: main thm: wellposed-HHF}
Assume that 
\begin{itemize}
\item $\gamma$ and $p$ satisfy (B0) and (B3) when $n\geq 3$; or
\item when $n=1,2$,  $\gamma$ and $p$ satisfy (A) or satisfy (B0)-(B1) for $n=1$ or  satisfy (B0) and (B2) for $n=2$. 
\end{itemize}
Let  
$$
u_0=w_0+\p\quad \text{with}\quad w_0\in \cH^{2-\frac{2}{p}+\delta, \gamma+2-\frac{2}{p}+\delta}_p(\M,\sN-\p) \quad \text{and} \quad \p\in\sN
$$ 
for an arbitrary positive constant $\delta$. Then \eqref{S4: HHF} with initial condition $u_0$ has a unique solution
\begin{align*}
u\in &L_p(J_T, \cH^{2,\gamma+2}_p(\M,\R^L)\oplus \R^L)\cap H^1_p(J_T,\cH^{0,\gamma}_p(\M,\sN) )\\
&\cap BC(J_T, \cH^{2-\frac{2}{p}-\varepsilon, \gamma+2-\frac{2}{p}-\varepsilon}_p(\M,\R^L)\oplus\R^L)
\end{align*}
on $J_T=[0,T)$ for some $T>0$ and any $\varepsilon>0$. Moreover, for any $\alpha>0$ and $k\in\Nz$
\begin{equation}
\label{S4: sol-reg 1}
u\in H^k_p((\alpha,T), \cH^{k+2,\gamma+2}_p(\M,\R^L)\oplus \R^L)\cap H^{k+1}_p((\alpha,T),\cH^{k,\gamma}_p(\M,\sN) )
\end{equation}
and
\begin{equation}
\label{S4: sol-reg 2}
u\in BC^\infty((\alpha,T), \cH^{\infty, \gamma+2-\frac{2}{p}-\varepsilon}_p(\M,\R^L)\oplus\R^L)\cap C^\infty((0,T),\cH^{\infty, \gamma+2}_p(\M,\R^L)\oplus\R^L).
\end{equation}
\end{theorem}
\begin{proof}
We have already shown the existence of a unique solution in
$$
L_p(J_T, \cH^{2,\gamma+2}_p(\M,\R^L)\oplus \R^L)\cap H^1_p(J_T,\cH^{0,\gamma}_p(\M,\sN) ).
$$
Based on the argument on  \cite[page~201]{PruSim16} and \eqref{S4: A-reg}, we can conclude that for any $\alpha\in (0,T)$ and $k\in \Nz$
$$
u\in H^k_p((\alpha,T), \cH^{2,\gamma+2}_p(\M,\R^L)\oplus \R^L)\cap H^{k+1}_p((\alpha,T), \cH^{0,\gamma}_p(\M,\sN) ).
$$
In view of Remark~\ref{S4: HHF2 higher regularity}, we can apply a bootstrapping argument to improve the spatial 
regularity to 
$$
u\in  H^k_p((\alpha,T), \cH^{l+2,\gamma+2}_p(\M,\R^L)\oplus \R^L) \cap H^{k+1}_p((\alpha,T), \cH^{l,\gamma}_p(\M,\sN) )
$$
for arbitrary $ l\geq 0$. 
Utilizing \cite[Chapter III: Theorem~4.10.2]{Ama95} and Lemma~\ref{S2.1: Sobolev-interpolation}, we thus have
$$
u\in BC^k((\alpha,T), \cH^{l_{\varepsilon,p}, \gamma_{\varepsilon,p}}(\M,\R^L)\oplus\R^L)
$$
for all $\alpha\in (0,T)$, $l\geq 0$, $k \in\Nz$ and $\varepsilon>0$ small. Therefore, we have obtained the extra temporal and spatial regularity in \eqref{S4: sol-reg 1} and \eqref{S4: sol-reg 2}.
\end{proof}



\section{\bf Global existence of the harmonic map heat flow into manifolds with nonpositive sectional curvature}
\label{Section 5}

In this section, we assume that the sectional  curvature $K_\sN$ of $\sN$ is nonpositive, and the following condition
\begin{mdframed}
\begin{itemize}
\item[(A')] When $n=1,2$, we assume that 
$$
4/p<\gamma<\frac{n+1}{2},\quad  p>6,
$$
and
$
\gamma<\sqrt{4-n-\lambda_1}-1.
$
\end{itemize}
\end{mdframed}
or (B0) and  (B$j$) in Section~\ref{Section 4}  hold for $j=\min\{3,n\}$. 
\begin{remark}
Assumption (A') is slightly stronger than (A), as the upper bound $\frac{n+1}{2}$ of $\gamma$ in (A') is smaller than $2$. We need this stronger upper bound  $\frac{n+1}{2}$ to derive an estimate for the nonlinear term $A_g(u)(\nabla u,\nabla u)$ in the space $L_p(J_T, \cH^{0,\gamma}_p(\M,\R^L))$ (cf. \eqref{S5: est A(u)}) by means of a gradient estimate of the solution $u$.
\end{remark}

Because of the compactness of $(\tilde{\M},\tilde{g})$, we can find a constant $C>0$ such that the Ricci curvature  $\Ric_{\tilde{g}}$ of $(\tilde{\M},\tilde{g})$
satisfies 
\begin{equation}
\label{S5: Ric bdd}
\Ric_{\tilde{g}}\geq -C.
\end{equation}

We will assume the initial condition $u_0$ satisfies
\begin{mdframed}
\begin{itemize}
\item[(IC)] $u_0=w_0+\p$, for some $\p\in\sN$, with $w_0\in  \cH^{2-\frac{2}{p}+\delta, \gamma+2-\frac{2}{p}+\delta}_p(\M,\sN-\p)$ and $ \|\nabla w_0\|_\infty<\infty  $. 
\end{itemize}
\end{mdframed}

Assume that $\mho$ is a geometric object on $(\M,g)$ defined in terms of the metric $g$. We use $\mho(\B)$ to  denote the corresponding object on $(\B,g_\B)$ defined with respect to $g_\B$.

Suppose that the local patches in $I\times\B$ are of the form $I\times O_j$, where $O_j$ are the local patches of $\B$.

In local coordinates, the Ricci curvature of $(\M,g)$ is of the form
$\Ric_g= R_{ik} dx^i \otimes dx^k$, where
\begin{equation}
\label{S5: local-Ricci}
R_{ik}=\partial_j \Gamma^j_{ki} -\partial_k \Gamma^j_{ji} + \Gamma^j_{jl} \Gamma^l_{ki} - \Gamma^j_{kl} \Gamma^l_{ji}.
\end{equation}
Here $\Gamma^j_{ki}$ are the Christoffel symbols of $g$. 

Let us first compute the Christoffel symbols of $g$ in $I\times \B$. 
In the following calculations, we will use $i,j,k,l$ to denote the subscripts of the local coordinates in $(\B,g_\B)$ and $x$ to denoted the coordinates in $I$ as usual. 
Careful computations yield
\begin{align*}
\Gamma^j_{ki}= \Gamma^j_{ki}(\B),\quad \Gamma^x_{ij}=-x g_{ij}(\B),\quad \Gamma^x_{xj}=\Gamma^i_{xx}=0,\quad \Gamma^i_{xj}=\frac{1}{x}\delta^i_j.
\end{align*}
Plugging the above expressions into \eqref{S5: local-Ricci}, we can obtain
$$
R_{ik}=R_{ik}(\B)-(n-1)g_{ik}(\B),\quad R_{xx}=0,\quad R_{xk}=0.
$$
Therefore, the Ricci curvature of $(\M,g)$ is of the form
$$
\Ric_g=\Ric_{g_\B} -(n-1)g_\B
$$
in $I\times \B$, and thus is globally bounded from below on $\M$.

To prove that  the solution $u$ to \eqref{S4: HHF} with initial value $u_0$ satisfying (IC)  is global. We argue by contradiction.
For, otherwise, we can assume that $u$ has its maximal existence time $0<T_\max<\infty$.

Let $e(u):=|\nabla u|_g^2$ be the energy density.  Now we can follow the argument in \cite[Lemma~5.3.3]{LinWang08} and prove that
$$
\partial_t e(u) -\Delta_g e(u)\leq M e(u) \ {\rm{on}}\  \M_{T_\max}
$$
for some $M>0$. It yields that
$$
(\partial_t -\Delta_g)(e^{-Mt}e(u))\leq 0 \ {\rm{on}}\  \M_{T_\max}. 
$$
Define
\begin{align*}
f_+=
\begin{cases}
f \quad &\text{if }f\geq 0;\\
0 &\text{if }f<0.
\end{cases}
\end{align*}
Let $v(x,t)=(e^{-Mt}e(u)(x,t) -\|e(u_0)\|_\infty)_+$ for $(x,t)\in \M_{T_\max}$.
Then
$$
(\partial_t -\Delta_g) v \leq 0  \ {\rm{on}}\  \M_{T_\max}.
$$
Since $v(0)\equiv 0$, multiplying the above inequality by $v(x,t)$ and integrating over $\M$ gives us
$$
(e^{-Mt}e(u)(x,t) - \|e(u_0)\|_\infty)_+ \equiv 0 \quad \text{for all }(x,t)\in \M_{T_\max}.
$$
This implies that 
$$
e^{-Mt}e(u)(x, t)) \leq \|e(u_0)\|_\infty \quad \text{for all }(x,t)\in  \M_{T_\max},
$$
and thus 
$$
|\nabla u|_g^2(x,t)\leq e^{M T_\max} \|e(u_0)\|_\infty \quad \text{for all }(x,t)\in\M_{T_\max}.
$$
We can infer from the above gradient estimate, Conditions  (A') and (B0) that
\begin{equation}
\label{S5: est A(u)}
A_g(u)(\nabla u, \nabla u) \in L_p(J_\max, \cH^{0,\gamma}_p(\M,\sN)).
\end{equation}
Here $J_{\max}=[0, T_{\max}]$.

Now it follows from \eqref{S3.1: MR-Lis},  \cite[Chapter III: Theorem~4.10.2]{Ama95} and Lemma~\ref{S2.1: Sobolev-interpolation}  that for any $\varepsilon>0$
\begin{align}
\label{S5: higher est}
\notag &\|u\|_{BC(J_\max, \cH^{2-\frac{2}{p}-\varepsilon, \gamma+2-\frac{2}{p}-\varepsilon}_p(\M,\R^L)\oplus\R^L)}\\
\notag \leq & \|u\|_{L_p(J_\max, \cH^{2,\gamma+2}_p(\M,\R^L)\oplus \R^L)\cap H^1_p(J_\max,\cH^{0,\gamma}_p(\M,\R^L) )}\\
\notag \leq & \|u_0\|_{\cH^{2-\frac{2}{p}+\delta, \gamma+2-\frac{2}{p}+\delta}_p(\M,\R^L)\oplus \R^L} + 
\|A_g(u)(\nabla u, \nabla u)\|_{L_p(J_\max, \cH^{0,\gamma}_p(\M,\R^L))}\\
\leq  & C
\end{align}
for some $C>0$.
From the above inequality, Theorem~\ref{S2.3: ptwise-mul}(i), Conditions  (A') and (B0), we can further infer that
$$
A_g(u)(\nabla u, \nabla u) \in L_p(J_\max, \cH^{1-\frac{2}{p}-\varepsilon,\gamma+\tilde{\delta}}_p(\M,\sN)).
$$
for some $\tilde{\delta}>0$. We can pick $\tilde{\delta}$ so small that (A') or (B0) and (B$j$) still hold for $j=\min\{3,n\}$ with $\gamma$ replaced by $\gamma+\tilde{\delta}$; 
and note that 
$$
u(\alpha)\in \cH^{2,\gamma+2}_p(\M,\R^L)\oplus \R^L \hookrightarrow \cH^{2-\frac{2}{p}+\delta^\prime, \gamma+2-\frac{2}{p}+\delta^\prime}_p(\M,\R^L)\oplus\R^L
$$ 
for all $\alpha\in (0,T_\max)$ and $\delta^\prime$ sufficiently small.

Choose $\tilde{\delta}<\delta^\prime$.
Now we can utilize the argument in \eqref{S5: higher est}  once more to obtain a higher order estimate in both  spatial regularity and the weight.
\begin{align*}
&\|u\|_{BC((\alpha, T_\max), \cH^{2-\frac{2}{p}+\tilde{\delta}-\varepsilon, \gamma+2-\frac{2}{p} + \tilde{\delta}-\varepsilon}_p(\M,\R^L)\oplus\R^L)}\\
\leq & \|u\|_{L_p((\alpha, T_\max), \cH^{2+\tilde{\delta},\gamma+2+\tilde{\delta}}_p(\M,\R^L)\oplus \R^L)\cap H^1_p((\alpha, T_\max),\cH^{\tilde{\delta},\gamma+\tilde{\delta}}_p(\M,\R^L) )}\\
\leq & \|u(\alpha)\|_{\cH^{2-\frac{2}{p}+\delta^\prime, \gamma+2-\frac{2}{p}+\delta^\prime}_p(\M,\R^L)\oplus \R^L} + 
\|A_g(u)(\nabla u, \nabla u)\|_{L_p((\alpha, T_\max), \cH^{\tilde{\delta},\gamma+\tilde{\delta}}_p(\M,\R^L))}\\
\leq  & C
\end{align*}
for some $C>0$ and arbitrary $\varepsilon>0$.

By means of Lemma~\ref{S2.1: Sobolev-iso}(ii), we can thus extract a bounded sequence 
$$\{u(t_k)\}_k \subset \cH^{2-\frac{2}{p}+\tilde{\delta}-\varepsilon, \gamma+2-\frac{2}{p} + \tilde{\delta}-\varepsilon}_p(\M,\R^L)\oplus\R^L \quad \text{with}\quad t_k\to T_\max^- $$
such that
$$
u(t_k) \to u^* \quad \text{in } \cH^{2-\frac{2}{p}+\delta, \gamma+2-\frac{2}{p}+\delta}_p(\M,\R^L)\oplus\R^L
$$
for some  $\delta>0$ and some $u^*\in \cH^{2-\frac{2}{p}+\delta, \gamma+2-\frac{2}{p}+\delta}_p(\M,\R^L)\oplus\R^L$.  
More precisely, it follows from Proposition~\ref{S2.2: Sobolev-embedding} that $u^*=w^*+z$ for some $z\in\sN$ and $w^*\in \cH^{2-\frac{2}{p}+\delta, \gamma+2-\frac{2}{p}+\delta}_p(\M,\sN-z)$. Because of the continuity of $u$ in the space 
$$
BC((\alpha, T_\max), \cH^{2-\frac{2}{p}+\tilde{\delta}-\varepsilon, \gamma+2-\frac{2}{p} + \tilde{\delta}-\varepsilon}_p(\M,\R^L)\oplus\R^L),
$$
the limit $u^*$ is independent of the choice of the sequence $\{u(t_k)\}_k$.

Therefore, we can apply Theorem~\ref{S4: main thm: wellposed-HHF} to extend $u$ beyond $T_\max$ smoothly. This leads to a contradiction. 

Let us summarize the above discussion in the following theorem.
\begin{theorem}
\label{S5: Global sol to HHF}
Suppose that the sectional curvature  of $\sN$ is nonpositive.
Assume that 
\begin{itemize}
\item $\gamma$ and $p$ satisfy (B0) and (B3) when $n\geq 3$; or
\item when $n=1,2$,  $\gamma$ and $p$ satisfy (A') or satisfy (B0)-(B1) for $n=1$ or  satisfy (B0) and (B2) for $n=2$. 
\end{itemize}
Let
$$
u_0=w_0+\p\quad \text{with}\quad \p\in\sN \quad \text{and} \quad w_0\in \cH^{2-\frac{2}{p}+\delta, \gamma+2-\frac{2}{p}+\delta}_p(\M,\sN-\p)
$$ 
for an arbitrary positive constant $\delta$; in addition $\|\nabla w_0\|_\infty<\infty$. Then \eqref{S4: HHF} with initial condition $u_0$ has a unique global solution
\begin{align*}
u\in &L_p([0,T), \cH^{2,\gamma+2}_p(\M,\R^L)\oplus \R^L)\cap H^1_p([0,T),\cH^{0,\gamma}_p(\M,\sN) )\\
&\cap BC([0,T), \cH^{2-\frac{2}{p}-\varepsilon, \gamma+2-\frac{2}{p}-\varepsilon}_p(\M,\R^L)\oplus\R^L)
\end{align*}
for any $\varepsilon>0$ and $T>0$. Moreover, for any $\alpha>0$ and $k\in\Nz$
$$
u\in H^k_p((\alpha,T), \cH^{k+2,\gamma+2}_p(\M,\R^L)\oplus \R^L)\cap H^{k+1}_p((\alpha,T),\cH^{k,\gamma}_p(\M,\sN) )
$$
and
$$
u\in BC^\infty((\alpha,T), \cH^{\infty, \gamma+2-\frac{2}{p}-\varepsilon}_p(\M,\R^L)\oplus\R^L)\cap C^\infty((0,\infty),\cH^{\infty, \gamma+2}_p(\M,\R^L)\oplus\R^L).
$$
\end{theorem}


\section{Acknowledgements}
Parts of the paper were completed while the first author was a Golomb Assistant Professor at Purdue University. 
He would like to thank the staff and faculty at Purdue University for providing the friendly environment. 
The first author would also like to express his gratitude to Prof. Elmar Schrohe for helpful discussions on the cone differential operators during his stay in Leibniz University Hanover in 2015.
The second author is partially supported by NSF.

\end{document}